\documentclass{article}
\usepackage{amssymb}
\usepackage{amsmath}
\usepackage[normalem]{ulem}

\newtheorem{theorem}{Theorem}[section]

\newtheorem{proposition}[theorem]{Proposition}
\newtheorem{corollary}[theorem]{Corollary}
\newtheorem{definition}{Definition}[section]
\newtheorem{preexample}{Example}[section]

\newtheorem{preremark}{Remark}
\newenvironment{remark}{\begin{preremark}\rm}{\end{preremark}}

\newenvironment{proof}
  {{\bf Proof:}}
  {\qquad \hspace*{\fill} $\Box$}%

\newcommand{\fg}{\mathfrak{g}}%

\newcommand{\Ad}{\operatorname{Ad}}%
\newcommand{\ad}{\operatorname{ad}}%
\newcommand{\rme}{\mathrm{e}}%
\newcommand{\DC}{\mathcal{D}}%
\newcommand{\XC}{\mathcal{X}}%
\newcommand{\R}{\mathbb{R}}%
\newcommand{\C}{\mathbb{C}}%

\begin{document}

\title{Linear flows on compact, semisimple Lie groups: stability, periodic orbits, and Poincar\'e-Bendixon's Theorem}

\author{S. N. Stelmastchuk\\ Universidade Federal do Paran\'{a}\\Jandaia do Sul, Brazil
,
\footnote{{\bf e-mail: }simnaos@gmail.com}
}

\maketitle

\begin{abstract}
	Our first purpose is to study the stability of linear flows on real, connected, compact, semisimple Lie groups. After, we study and classify periodic orbits of linear and invariant flows. In particular, we obtain a version of Poincar\'e-Bendixon's Theorem. As an application, we present periodic orbits of linear or invariant flows on $SO(3)$ or $SU(2)$, and we classify periodic orbits of a linear or invariant system on $SO(4)$.
\end{abstract}



\section{Introduction}

Let $G$ be a real, connected Lie group. A vector field $\XC$ on $G$ is called {\it linear} if its flow, which is denoted by $\varphi_t$, is a family of automorphisms of $G$. In this work, we assume that $G$ is a semisimple Lie group. Our wish is to study some aspects of stability of linear flow $\varphi_t$ and periodic orbits of linear and invariant flows. 

Our first task is to study the stability in a fix point of linear flow $\varphi_t$. In natural way, we follow ideas presented in classical literature of dynamical systems on Euclidian space (see for instance \cite{colonius}, \cite{hirsch} and \cite{robinson}). In \cite{santana}, Da Silva, Santana and Stelmastchuk show that a necessary and sufficient condition to asymptotically and exponential stability of $\varphi_t$ at identity $e$ is that $\XC$ is hyperbolic. However, if a linear vector field $\XC$ is hyperbolic, then $G$ is a nilpotent Lie group. Because this obstruction, we choose restrict our study on compact, semisimple Lie groups. 

Let $G$ be a real, connected, compact, semisimple Lie group. Consider a linear vector field $\XC$ on $G$ and its linear flow $\varphi_t$. We show that any fix point of linear flow $\varphi_t$ is stable (see Theorem \ref{stable}). Furthermore, we demonstrate that any periodic orbit of linear flow $\varphi_t$ is stable (see Theorem \ref{periodic}). Also, we proof that the derivation $\DC = -ad(\XC)$ associated to $\XC$ has only semisimple eigenvalues since the identity $e$ is stable. The last fact is the key to study periodic orbits of linear flow $\varphi_t$.

Our next purpose is to study periodic orbits of linear flows on compact, semisimple Lie groups. In semisimple Lie groups, for any linear vector field $\XC$ there is a right invariant vector field $X$ associated to it. Thus, our first step is to study the relation between periodic orbits of linear flow $\varphi_t$ and invariant flow $\exp(tX)$. In fact, we proof that they are equivalent. Done this, we show that an orbit of linear or invariant system is periodic if and only if the derivation $\DC$ of $\XC$ has as eigenvalues $0$ and $\mu = \pm \alpha i$ for an unique $\alpha \in \R$ (see Theorem \ref{teo1}). As a direct consequence, every orbit that is not a fix point of a linear flow $\varphi_t$ or invariant flow $\exp(tX)$ on a 3-dimensional, compact, semisimple Lie group $G$ is periodic. Other consequence is a version of Poincar\'e-Bendixon Theorem for linear or invariant flows. We remark that for a linear or invariant flow all orbits that is not a fix point of linear or invariant flows are periodic or not.

To end, we presents periodic orbits of linear and invariant flows on $SO(3)$ and $SU(2)$, and we classify periodic orbits on $SO(4)$ (see Theorem \ref{periodicSO(4)}).

This paper is organized as follows. Section 2 briefly reviews the notions of linear vector fields. Section 3 works with stability on compact, semisimple Lie groups. Section 4 develops results about periodic orbits. Finally, section 5 applies previous results on compact, semisimple Lie groups $SO(3)$, $SU(2)$ and $SO(4)$. 

\section{Linear vector fields}

Let $G$ be a connected Lie group and let $\fg$ denote its Lie algebra. We call a vector field $\XC$ linear if its flow $(\varphi_t)_{t\in\R}$ are automorphisms of Lie group. It is known there is a derivation $\DC$ associated to $\XC$, which is given by
\[
  \DC(Y) = - [\XC,Y], \, Y \in \fg. 
\] 
In Euclidian case, it is true that $[Ax,b] = - Ab$. As a consequence, the derivation $\DC$ coincide with linear map induced by $A$; then, the dynamical system 
\[
  \dot{g} = \XC(g), \ \ g \in G,
\]
is a generalization of dynamical system on $\R^{n}$ given by
\[
  \dot{x} = Ax,
\]
where $A \in \R^{n \times n}$ and $x \in \R^n$. 

Da Silva, in \cite{dasilva}, write 
\[
  \fg^+=\bigoplus_{\alpha; \mathrm{Re}(\alpha)>0}\fg_{\alpha}, \;\;\;\;\fg^0=\bigoplus_{\alpha; \mathrm{Re}(\alpha)=0}\fg_{\alpha}, \;\;\;\mbox{ and }\;\;\;\fg^-=\bigoplus_{\alpha; \mathrm{Re}(\alpha)<0}\fg_{\alpha}
\]
where $\alpha$ are eigenvalues of the derivation $\DC$ such that 
\[ 
  \fg =\fg^+\oplus \fg^{0} \oplus \fg^{-} \ \ \mbox{and} \ \ [\fg_{\alpha},\fg_{\beta}] = \fg_{\alpha + \beta}
\]
with $\alpha + \beta = 0$ if the sum is not an eigenvalue. Let us denote by $G^+$, $G^0$ and $G^-$  $\varphi_t$-invariant, connected Lie subgroups of Lie algebras $\fg^+$, $\fg^0$ and $\fg^-$, respectively. Lie subgroups $G^+$, $G^0$ and $G^-$ are called unstable, central and stable groups associated to $\varphi_t$, respectively. 

For the convenience of the reader we resume some facts about a linear vector field $\XC$ and its flow $\varphi_t$. The proof of these facts can be found in \cite{cardetti}. 
\begin{proposition}\label{linearproperties}
  Let $\XC$ be a linear vector field and let $\varphi_t$ denote its flow. The following assertions hold:
  \begin{description}
    \item{(i)} $\varphi_t$ is an automorphism of Lie groups for each $t$;
	  \item{(ii)} $\XC$ is linear iff $\XC(gh) = R_{h*}\XC(g) + L_{g*}\XC(h)$; 
  	\item{(iii)} $(d\varphi_t)_e = e^{t\DC}$ for all $t \in \R$.
	\end{description}
\end{proposition}

\section{Stability of the linear flow}
  

Let $G$ be a semisimple Lie group and $\XC$ a linear vector field on $G$. In this section, our wish is to study stability of the linear flow $\varphi_t$ that is the solution of the differential equation on $G$ given by  
\begin{equation}\label{odelinear}
  \dot{g} = \XC(g).
\end{equation}

Being $G$ semisimple, there is a right invariant vector field $X$ such that $\XC = X + I_*X$, where $I_{*}X$ is the left invariant vector field associated to $X$ and  $I_*$ is the differential of inverse map $\mathfrak{i}(g) = g^{-1}$ (more details is founded in \cite{sanmartin}). From this it follows that the linear flow can be written as   
\[
  \varphi_t(g) = \exp(tX).g.\exp(t (I_*X)), \, \forall \, g \in G.
\]
According to above expression, we have that the identity $e$ is a fix point for the linear flow. However, it may to exist other fix point.

\begin{proposition}
  If $g$ is a point belongs to center of $G$, then $g$ is a fix point of the linear flow $\varphi_t$.
\end{proposition}
\begin{proof}
  Let $g$ be a point in the center of $G$. Then, for all $t \in \R$,
	\[
	  \varphi_t(g) = \exp(tX)\cdot g \cdot \exp(-tX) =  \exp(tX) \cdot \exp(-tX) \cdot g = g,
	\]
	which is the desired conclusion.
\end{proof}

Our next step is to presents the hyperbolic concept to linear vector fields. We remember that the stability in Euclidian space is obtained if a dynamical system is hyperbolic (see for instance \cite{robinson}). As one can see in \cite{santana}, it is also true if $\XC$ is hyperbolic on a Lie group $G$. 

\begin{definition}
  Let $\mathcal{X}$ be a linear vector field on a Lie group $G$. We call $\XC$ {\bf hyperbolic} if its associated derivation $\DC$ is hyperbolic, that is, $\DC$ has no eigenvalues with zero real part.
\end{definition}

Let $\XC$ be a hyperbolic linear vector field on a semisimple Lie group $G$. Then $\DC$ has no eigenvalues with zero real part. Denoting by $\fg_{\alpha}$ the generalized eigenspace associated with an eigenvalue $\alpha$ of $\DC$ we get   
\[
  [\fg_{\alpha}, \fg_{\beta}]\subset\fg_{\alpha+\beta},
\]
where $\alpha+\beta$ is an eigenvalue of $\DC$ and zero otherwise (see for instance Proposition 3.1 in \cite{sanmartin}). Since $\dim G<\infty$, it implies that the Lie algebra $\fg$ is nilpotent. In consequence, $G$ is nilpotent. We have thus proved

\begin{proposition}\label{noexistence}
  There no exists hyperbolic linear vector field on semisimple Lie groups. 
\end{proposition}

We now begin the study of the stability of  linear flows on semisimple Lie groups. Firstly, we remember some concepts of stability. 
\begin{definition}
  Let $g\in G$ be a fixed point of the linear vector field $\mathcal{X}$. We call $g$ 
  \begin{itemize}
    \item[1)]{\bf stable} if for all $g$-neighborhood $U$ there is a $g$-neighborhood $V$ such that $\varphi_t(V)\subset U$ for all $t\geq 0$;
    \item[2)] {\bf asymptotically stable} if it is stable and there exists a $g$-neighborhood $W$ such that $\lim_{t\rightarrow\infty}\varphi_t(x)=g$ whenever $x\in W$;
    \item[3)] {\bf exponentially stable} if there exist $c, \mu$ and a $g$-neighborhood $W$ such that for all $x\in W$ it holds that
    $$\varrho(\varphi_t(x), g)\leq c\rme^{-\mu t}\varrho(x, g),\;\;\;\;\mbox{ for all }\;\;t\geq 0;$$
    \item[4)] {\bf unstable} if it is not stable.
  \end{itemize}
\end{definition}

Since property 3) is local, it does not depend of the metric on $G$. Because of this, we will assume from now on that $\varrho$ is a left invariant Riemmanian metric.

In order to characterize the stability, let us work with the Lyapunov exponent. We follow \cite{santana} in assuming that the Lyapunov exponent can be written as  
\[
  \lambda(e,v) = \limsup_{t \to \infty} \frac{1}{t} \rm log(\|e^{tD}(v)\|),
\]
where $v$ is $\fg$ and the norm $\| \cdot \|$ is given by the left invariant metric.

We will use $\lambda_1, \ldots, \lambda_k$ to denote $k$ distinct values of the real parts eigenvalues of the derivation $\DC$. Then, the Lie algebra $\fg$ can be written as 
\[
  \fg=\bigoplus_{i=1}^k\fg_{\lambda_i} \;\;\mbox{ where }\;\;\fg_{\lambda_i}:=\bigoplus_{\alpha; \mathrm{Re}(\alpha)=\lambda_i}\fg_{\alpha}.
\]
Furthermore, from Theorem 4.2 in \cite{santana} we see that 
\begin{equation}\label{lyapunovroots}
  \lambda(e, v)= \lambda \;\;\;\Leftrightarrow\;\;\;v\in\fg_{\lambda}:=\bigoplus_{\alpha; \mathrm{Re}(\alpha)=\lambda}\fg_{\alpha}.
\end{equation}

Using the Lyapunov exponent we make a first result about stability of the linear flow $\varphi_t$.

\begin{theorem}
  For any linear vector field $\XC$ on a semisimple Lie group $G$, any fix point is neither asymptotically nor exponentially stable to the linear flow $\varphi_t$. 
\end{theorem}
\begin{proof}
	We first observe that Lyapunov exponents satisfy the following: $\lambda(g,v) = \lambda(e,v)$ for each $v \in \fg$. We need only consider the assertion at identity $e$. Suppose, contrary to our claim, that the identity $e$ is either asymptotically or exponentially stable. By Theorem 4.5 in \cite{santana}, it follows that all Lyapunov exponents of $\DC$ are negatives. From (\ref{lyapunovroots}) it follows that any eigenvalue of $\DC$ has the real part negative. It means that $\XC$ is hyperbolic, and this contradicts Proposition \ref{noexistence}.
\end{proof}

Despite any fix point is neither asymptotically nor exponentially stable, they are stable if $G$ is compact and semisimple as we will show. For this purpose, we begin by introducing an appropriate metric on $G$. 

Let $G$ be a compact and semisimple. Then, we know that the Cartan-Killing form is negative defined. Thus, we adopt the metric $<,>$ given by negative of the Cartan-Killing form on $\fg$. Since $<,>$ satisfies 
\[
  <\Ad(g)X,\Ad(g)Y> = <X,Y>, \,\, \forall \, g \in G \,\, \mbox{and}\,\, X,Y \in \fg,
\]
it follows that $<,>$ is an invariant Riemmanian metric on $G$ associated to Cartan-Killing form (see \cite{arvani} for more details). From now on we make the assumption: every compact, semisimple Lie group is equipped with the Riemannian metric given by Cartan-Killing form. 

Adopting these invariant metric and using Lyapunov exponent leads us a algebraic characterization of linear vector fields on compact, semisimple Lie groups, namely, its eigenvalues have real part null. 

\begin{proposition}\label{central}
  Let $\XC$ be a linear vector field on a compact, semisimple Lie group $G$. Then, $G$ is the central group of the linear flow $\varphi_t$.
\end{proposition}
\begin{proof}
  We begin by writing $\XC = X + I_*(X)$ with $X \in \fg$. It is clear that $\DC = \ad(X)$. Then, for any $v \in \fg$ we have 
	\[
	  \|e^{t\DC}v\| = \|e^{t\, \ad(X)}v\| = \|\Ad (\exp(tX))v\|= \|v\|,
	\]
	where we used the $\Ad$-invariance of metric at last equality. Thus, Lyapunov exponents can be written as 
	\[
	  \lambda(e,v) = \limsup_{t \to \infty} \frac{1}{t} \rm log(\|e^{t\DC}v\|) = \limsup_{t \to \infty} \frac{1}{t} \rm log(\|v\|) = 0.
	\]
	Therefore, $\lambda_1= \ldots= \lambda_k = 0$. Using the relation (\ref{lyapunovroots}) we conclude that $\fg =\fg_0$. Since $G$ is connected, $G= G_0$. It means that $G$ is the central group associated to linear flow $\varphi_t$.
\end{proof}

Despite Proposition above is presented in \cite{dasilva}, we want to prove this because our proof is done by dynamical concepts instead of algebraic concepts, as done there.

Our next step is to verify if the linear flow $\varphi_t$ satisfies some metric property. To do this, we recall what is a Riemmanian distance. Let $(M,g)$ be a Riemmanian manifold, a Riemmanian distance is $\rho$ associated to $g$ is defined by 
\[
  \rho(x,y) = \displaystyle\inf_{\sigma}\{\int_0^1 g( \dot{\sigma}(s), \dot{\sigma}(s))^{1/2} ds\},
\]
where the infimun is taken over all smooth curves $\sigma$ such that $\sigma(0) = x$ and $\sigma(1) = y$. 

\begin{proposition}\label{isometry}
  Let $\XC$ be a linear vector field on a compact, semisimple Lie group $G$. Then, $\varphi_t$ is an isometry for all $t$.
\end{proposition}
\begin{proof}
  For any $g, h \in G$ and any $t \in \R$ the Riemmanian distance between $\varphi_t(g)$ and $\varphi_t(h)$ is taken over all smooth curves $\sigma_t(s)$ such that 
	\[
	  \sigma_t(0) = \varphi_t(g) \ \ \mbox{and}\ \ \sigma_t(1) = \varphi_t(h). 
	\]	
	It is clear that smooth curves $\phi_t^{-1}\sigma_t(s)$ satisfy 
	\[
	  \varphi_t^{-1}\circ \sigma_t(0)=g\ \ \mbox{e}\ \ \varphi_t^{-1}\circ \sigma_t(1) = h.
	\]
	Using the invariance of metric $g$ we obtain 
	\[
	  \int_0^1 g( \sigma'(s), \sigma'(s) )^{1/2} ds  = \int_0^1 g((\varphi_t^{-1}\circ \sigma)'(s), (\varphi_t^{-1}\circ \sigma)'(s) )^{1/2} ds.
	\]
  Taking the infimun over all smooth curves joining $\varphi_t(g)$ and $\varphi_t(h)$ yields 
	\[
	  \rho(\varphi_t(g), \varphi_t(h)) =  \rho(g,h),
	\]
	which shows that $\varphi_t$ is an isometry on $G$ to the invariant metric given by Cartan-Killing form.
\end{proof} 

Before our next result, we need to introduce some notations. For $r>0$ we will denote an sphere of radius $r$ with center $g$ by $\mathbb{S}_r(g) = \{ x \in G: \rho(x,g) =r\}$ and an open ball of radius $r$ with center $g$ by  $ B_r(g) =\{ x \in G; \rho(x,g) < r\}$.

\begin{proposition}\label{sphere}
  If $G$ is a compact,semisimple Lie group, then for each $g \in G$ the linear flow $\varphi_t(g)$ is in an sphere. 
\end{proposition}
\begin{proof}
  We first choose an arbitrary point $g \in G$ and write $r=\rho(g,e)$. Then, 
	\[ 
	  \rho(\varphi_t(g),e) = \rho(\varphi_t(g), \varphi_t(e)) = \rho(g,e)=r, \, \,  \forall\,\, t.
	\]
	It means that $\varphi_t(g) \in \mathbb{S}_r$ for all $t$, and the proof is complete. 
\end{proof}

A first consequence of Propostion above is about $\omega$-limit and $\alpha$-limit sets.

\begin{corollary}
	If $G$ is a compact, semisimple Lie group, then $\omega$-limit and $\alpha$-limit sets of $g$ are in spheres.
\end{corollary}

We can now to prove our main result of our section.

\begin{theorem}\label{stable}
	Let $G$ be a compact, semisimple. Then, any fix point of linear flow $\varphi_t$ is an stable point. 
\end{theorem}
\begin{proof}
	We begin by fixing an arbitrary fix point $g$ of $G$. We also remember that a Riemmanian distance induces the topology of Riemmanian manifold. So it is sufficient to consider as neighborhoods of $g$ open balls $B_r(g)$, where $r>0$ is arbitrary. Choose $r_0>0$ such that $r_0 \leq r$ and consider the ball $B_{r_0}(g)$. Taking any $y \in B_{r_0}(g)$ we see that 
	\[
	  \rho(\varphi_t(y),g) = \rho(y,g) < r_0 \leq r,
	\]
	where we used Proposition (\ref{isometry}) at first equality. It shows that $\varphi_t(B_{r_0}(g)) \subset B_{r}(g)$. Consequently, by definition, $g$ is a stable point to the linear flow $\varphi_t$.  
\end{proof}

Hereafter, we give a characterization of derivations on compact, semisimple Lie groups. Before we need to introduce some concepts. Following \cite{colonius}, if for an eigenvalue $\mu$ all complex Jordan blocks are one-dimensional, i.e., a complete set of eigenvectors exists, it is called semisimple. Equivalently, the corresponding real Jordan blocks are one-dimensional if $\mu$ is real, and two-dimensional if $\mu$, $\mu$ and $\bar{\mu} \in \mathbb{R}$. 

\begin{theorem}\label{derivation}
  On a compact, semisimple Lie group $G$, every derivation has only semisimple eigenvalues.
\end{theorem}
\begin{proof}
  Let $\DC$ be a derivation on $G$. From Theorem \ref{stable} we see that $e$ is a stable point of the linear flow $\varphi_t$ associated to $\DC$. Since $(d\varphi_t)_e = e^tD$, it follows that the linearization of $\dot{g} = \XC(g)$ is $X = \DC X$. Being $exp$ local diffeomorphism and $e$ stable, it follows that $0$ is stable. From Proposition \ref{central} we know that eigenvalues of $\DC$ has real part null. Then, Theorem 1.4.10 in \cite{colonius} assures that every eigenvalue of $\DC$ is semisimple, which gives the proof. 
\end{proof}

Theorem above is fundamental to study periodic orbits of linear flows. 

To end this section, we study the stability of periodic orbits to linear flow $\varphi_t$. We begin by recalling what means a stable periodic orbit. A periodic orbit $\Gamma$ of the linear flow $\varphi_t$ is stable if for each open set $V$ that contains $\Gamma$, there is an open set $W \subset V$ such that every solution, starting at a point in $W$ at $t = 0$, stays in $V$ for all $t \geq 0$.

Before to presents our next result, we need to introduce the following notation. Take $g \in G$ and  consider the orbit $\varphi_t(g)$.Write $Tube_r (\varphi_t(g)) = \{ h \in G: \rho(h,\varphi_t(g))<r$ for some $t \}$  for any $r>0$,.

\begin{proposition}
	Let $\XC$ be a linear vector field on a compact, semisimple Lie group $G$. If $h \in Tube_r(\varphi_t(g))$, then $\varphi_s(h) \in Tube_r(\varphi_t(g))$ for any $s \in \R$.
\end{proposition}
\begin{proof}
	Let $h \in Tube_r(\varphi_t(g))$. Then, for some $t$ we have $\rho( h, \varphi_t(g))< r$. From Proposition \ref{isometry} it follows that 
	\[
		\rho(\varphi_s(h), \varphi_{t+s}(g)) = \rho(h, \varphi_{t}(g)) <r,
	\]
	which implies that $\varphi_s(h) \in Tube_r(\varphi_t(g))$.
\end{proof}

\begin{theorem}\label{periodic}
	Let $\XC$ be a linear vector field on a compact, semisimple Lie group $G$. Then every periodic orbit is stable.
\end{theorem}
\begin{proof}
	Let $g \in G$ such that $\varphi_t(g)$ is a periodic orbit of linear flow $\varphi_t$. We consider a open set $V$ such that $\varphi_t(g) \subset V$. Take $r_0 = \inf \{ r : B_r(\varphi_t(g)) \subset U, \ \forall\,\, t\geq 0 \}$. Thus it is sufficient to take $U =Tube_{r_0}(\varphi_t(g))$ and to apply Proposition above. 
\end{proof}

\section{Periodic orbits}

In this section, we study periodic orbits of a linear flow in a compact, semisimple Lie group $G$. The key of our study is Theorem \ref{derivation} because it describes all eigenvalues of any derivation on $G$.

We begin by recalling that a linear vector field $\XC$ can be written as $\XC = X + I_*X$, where $X$ is a right invariant vector field, $I_{*}X$ is the left invariant vector field associated to $X$, and $I_*$ is the differential of inverse map $\mathfrak{i}(g) = g^{-1}$. In this way, we can rewrite the differential equation (\ref{odelinear}) as 
\[
  \dot{g} = X(g) + (I_*X)(g).
\]
It implies that there exists a relation between flows of linear dynamical system $\dot{g} = \XC(g)$ and of invariant one $\dot{g} = X(g)$. In fact, a direct accounts shows that, for all $g \in G$, $\varphi_t(g)$ is solution of (\ref{odelinear}) if, and only if, $\varphi_t(g)\cdot \exp(tX)$ is solution of $\dot{g} = X(g)$. It suggests us that there exists a relation between periodic orbits of invariant and linear system. Therefore, our next step is to investigate this fact.

\begin{proposition}\label{equivalentperiodic}
	Let $\XC$ be a linear vector field on a compact, semisimple Lie group $G$. The following sentences are equivalent:
	\begin{description}
		\item[i)] for every $g$, the invariant flow $\exp(tX)g$ is periodic;
		\item[ii)]$e$ is a periodic point of invariant flow $\exp(tX)$;
		\item[iii)] any $g \in G$ is a periodic point of linear flow $\varphi_t$;
		\item[iv)] for every $g$, $Ad(g)$ is a periodic point of the flow $e^{tD}$.
	\end{description} 
\end{proposition}
\begin{proof}
  i) $\Leftrightarrow$ ii)
	If $\exp(tX)g$ is periodic, then $e$ is a periodic point of the curve $\exp(tX)$. On contrary, suppose that $e$ is a periodic point of the flow $\exp(tX)$, that is, there is a $s>0$ such that $\exp((t+s) X) = \exp (tX)$. Then, for any $g \in G$
	\[
	  \exp((t+s)X)g = (\exp((t+s)X)\cdot e) \cdot g = \exp(tX)g.
	\]
  ii) $\Leftrightarrow$ iii)
  If $e$ is a periodic point of the flow $\exp(tX)$, then there is a $s>0$ such that $\exp(tX) = \exp((t+s)X)$. Thus
	\[
	  \varphi_{t+s}(g) = \exp((t+s)X)\cdot g \cdot \exp(-(t+s)X) = \exp(tX)\cdot g \cdot \exp(-tX) = \varphi_t(g).
	\]
	On contrary, if every $g$ is a periodic point of $\varphi_t$, then it is clear that $\exp(tX)$ is periodic.\\
	i) $\Leftrightarrow$ iv)
	Since $G$ is a semisimple Lie group, it follows 
	\[
	  \Ad(\exp(tX)\cdot g) = e^{t\ad(X)} \Ad(g) = e^{tD}\Ad(g).
	\]
	We thus get the equivalence. 
\end{proof}

The interest of the Proposition above is that periodic orbits of linear or invariant flows are equivalents on compact, semisimple Lie groups.

\begin{theorem}\label{teo1}
	Let $G$ be a compact, semisimple Lie group. It is equivalent
	\begin{description}
		\item[i)] if there exists a periodic orbit for the linear flow $\varphi_t$; or 
		\item[ii)] if there exists a periodic orbit for the right invariant flow $\exp(tX)$; 
		\item[iii)] the derivation $\DC$ of $\XC$ has only eigenvalues of the form $0$ or $\mu = \pm \alpha i$ with unique $\alpha \in \mathbb{R}^*$.
	\end{description}
	Furthermore, if there exists a periodic orbit, then its period is $T= 2\pi/\alpha$.
\end{theorem}
\begin{proof}
	We first observe that i) is equivalent to  ii) by Proposition \ref{equivalentperiodic}. We are going to show that ii) is equivalent to iii). For this, it is sufficient to consider $e$ as a periodic point to the flow $\exp(tX)$ with period $T>0$. Then, for all $t \in \mathbb{R}$,
	\[
		\exp((t+T)X) = \exp(tX) \Leftrightarrow \exp(TX) = e \Leftrightarrow e^{T\DC} = Id.
	\]
	Take the Jordan form $J$ of $\DC$. A simple account shows that $e^{TJ} =Id$.  Since any eigenvalues of $\DC$ is semisimple, its real Jordan Block has dimension 1 or 2 if it is real or complex, respectively. If $0$ is eigenvalue of $\DC$, then its real Jordan block is written as $J_0 =[0]$. Thus $e^{tJ_0}$ is constant. It implies that in direction of $0$ the $e^tJ$ is constant. Consequently, solutions associated to $0$ are trivially periodic. Suppose that there are eigenvalues different of $0$. From Proposition \ref{central} they are of the form $\pm\alpha i$ and $\pm\beta i$. By Theorem \ref{derivation}, its real Jordan blocks are, respectively,  
	\[
		\left(
	  \begin{array}{cc}
			\cos(t\alpha) & -\sin(t\alpha) \\
			\sin(t\alpha) & \cos(t\alpha)
		\end{array}
		\right)
		 \ \ \mbox{and} \ \
		\left(
	  \begin{array}{cc}
			\cos(t\beta) & -\sin(t\beta) \\
			\sin(t\beta) & \cos(t\beta)
		\end{array}
		\right).
	\]
	As $e^{TJ} =Id$ we have $T = \frac{2\pi}{\alpha}$ and $T = \frac{2\pi}{\beta}$. This clearly forces $\alpha = \beta$. Consequently, eigenvalues different of $0$ are only $\alpha i$ and $-\alpha i$.
	
	Reciprocally, suppose that $\DC$ has only eigenvalues $0$  or $\pm\alpha i$. In the case of $0$ as discussed above the solution is constant. In the case of an eigenvalue $\pm \alpha i$ with $\alpha \neq 0$. Being $\pm \alpha i$ semisimple, every real Jordan block associated to it has the dimension two and the solution applied at this block gives the following matrix 
	\[
		\left(
	  \begin{array}{cc}
			\cos(t\alpha) & -\sin(t\alpha) \\
			\sin(t\alpha) & \cos(t\alpha)
		\end{array}
		\right).
	\]
	It is clear that matrix above is periodic with period $T= \frac{2 \pi}{\alpha}$. It implies that $e$ is periodic point of $e^{TJ}$ with period $T= \frac{2\pi}{\alpha}$, which is equivalent $Ad(e) = e$ to be periodic point of $e^{TD}$. Consequently, by Proposition \ref{equivalentperiodic}, the right invariant flow $\exp(tX)$ is periodic with period of $T= \frac{2 \pi}{\alpha}$.
\end{proof}

\begin{remark}
	The theorem above shows that all orbit that is not a fix point of a linear or invariant flows are periodic or not.
\end{remark}

\begin{corollary}\label{periodorbits3}
	Let $G$ be a compact, semisimple Lie group with dimension 3. Then 
	\begin{description}
		\item [i)] every orbit of some invariant flow $\exp(tX)$ is periodic;
		\item [ii)] every orbit of some linear flow $\varphi_t$ is periodic.
	\end{description}
\end{corollary}
\begin{proof}
	It is sufficient to observe that the derivation $\DC=\operatorname{ad}(X)$ has only eigenvalues $0$, $\alpha i$, and $ -\alpha i$ with $\alpha \in \R^*$. 
\end{proof}

An immediate consequence of Theorem  and is version of Poincar\'e-Bendixson's Theorem for compact, semisimple Lie groups.

\begin{theorem}[Poincar\'e - Bendixson]
	Let $G$ be a compact, semisimple Lie group. If a derivation $\DC$ has only eigenvalues $0$ or $\pm \alpha i$, if $\Omega$ is a nonempty compact $\omega$-limit set for the linear flow $\varphi_t$, and if $Ω$ does not contain a fix point point of $\varphi_t$, then $\Omega$ is a periodic orbit.
\end{theorem}

\section{Applications}

In this section, our wish is to study the periodic orbits on compact, semisimple Lie groups of lower dimension. In fact, we are interested to describe linear flows, that are periodic orbits, on $SO(3)$ and $SU(2)$, and we are interested to classify the periodic orbits on $SO(4)$. 

\subsection{Linear flows on $SO(3)$ and $SU(2)$}

Our first case is to study  the linear flow on stability on the orthogonal group 
\[
  SO(3) = \{g \in \R^{3\times 3}: gg^{T} = 1, \, \det g = 1 \}.
\] 
It is well known that its Lie algebra is 
\begin{eqnarray*}
  \mathfrak{so}(3) & = &
	\left\{ \left[
	\begin{array}{ccc}
		0  & -z & y \\
		z  & 0  & -x \\
		-y & x  & 0
	\end{array}
	\right]:
	x,\, y, \, z \in \R
	\right\}.
\end{eqnarray*}

Let $\XC$ be a linear vector field on $SO(3)$. Then, there exists a right invariant vector field $X$ such that $\XC = X + I_*X$. A direct calculus shows that eigenvalues of $\DC = \ad(X)$ are 
\[
  \left\{0,-\sqrt{-x^2-y^2-z^2}, \sqrt{-x^2-y^2-z^2}\right\}.
\]

Write $\lambda_1 = -\sqrt{-x^2-y^2-z^2}$ and $\lambda_2 = \sqrt{-x^2-y^2-z^2}$. Then, using functional calculus we obtain 
\[
  \exp(tX) = \frac{\cosh(t \lambda_1)-1}{ \lambda_1^2}X^2+ \frac{\sinh(t\lambda_1)}{\lambda_1}X + Id.
\]
From this, it is possible to give the solution of the linear flow $\varphi_t$ on $SO(3)$.

\begin{proposition}
  Let $\XC$ be a linear vector field on $SO(3)$. The solution of linear flow $\varphi_t(g)$ associated to $\XC$ is  
	\[
    \left(\frac{\cosh(t \lambda_1)-1}{ \lambda_1^2}X^2+ \frac{\sinh(t\lambda_1)}{\lambda_1}X + Id\right) \cdot g \cdot \left(\frac{\cosh(t \lambda_2)-1}{ \lambda_2^2}X^2+ \frac{\sinh(t\lambda_2)}{\lambda_2}X + Id\right),
	\]
	where $X$ is the right invariant vector field associated to $\XC$ and $\lambda_1 = -\sqrt{-x^2-y^2-z^2}$ and $\lambda_2 = \sqrt{-x^2-y^2-z^2}$.
\end{proposition}

Corollary \ref{periodorbits3} now assures the characterization of periodic orbits of the linear flow $\varphi_t$.

\begin{proposition}
	On $SO(3)$, 
\begin{description}
		\item [i)] every orbit of some invariant flow $\exp(tX)$ is periodic;
		\item [ii)] every orbit of some linear flow $\varphi_t$ is periodic.
	\end{description}
	Furthermore, the period is $T = 2\pi /\sqrt{x^2+y^2+z^2}$. 
\end{proposition}

Our other case is the unitary group $SU(2)$, which is a matrix group given by
\[
  \mathrm{SU}(2) = \{ g \in \C^{2\times 2} : gg^T = 1, \,\mathrm{det} g = 1\}.
\]
The Lie algebra associated to $SU(2)$ is described as
\[
  \mathfrak{su}(2) =  
	  \left\{ \left[
	  \begin{array}{cc}
		  \frac{i}{2}x   & \frac{1}{2}(iz+y) \\
		  \frac{1}{2}(iz-y) & -\frac{1}{2}x \\
		\end{array}
		\right]:
		x,\, y, \, z \in \R
		\right\}.
\]

Let $\XC$ be a linear vector field on $SU(2)$ and $X$ the right invariant vector field associated to it. In analogous way to case $SO(3)$, it is easily to see that eigenvalues of a derivation $\DC = \ad(X)$ are 
\[
  \left\{0,-\sqrt{-x^2-y^2-z^2}, \sqrt{-x^2-y^2-z^2}\right\}.
\]
In consequence,

\begin{proposition}
	On $SU(2)$, 
	\begin{description}
		\item [i)] every orbit of some invariant flow $\exp(tX)$ is periodic;
		\item [ii)] every orbit of some linear flow $\varphi_t$ is periodic.
	\end{description}
	Furthermore, the period is $T = 2\pi /\sqrt{x^2+y^2+z^2}$. 
\end{proposition}

\subsection{periodic orbits on $SO(4)$}

In this subsection, our wish is to give a condition for orbits of invariant or linear flow be periodic or not on $SO(4)$. Let $\mathfrak{so}(4)$ be the Lie algebra of $SO(4)$ given by
\[
	\left\{ \left[
		\begin{array}{cccc}
		0 & -x & -y & -z\\
		x & 0  & -u & -v \\
		y & u  &  0 & -w \\
		z & v  &  w &  0 
	\end{array}
	\right]:
	x,\, y,\, z, \, u, \, v, \, w \in \R
	\right\}.
\]
Consider the basis $\beta$ for $\mathfrak{so}(4)$ that consists of $4 \times 4$ matrices $e_{12},e_{13},e_{14},e_{23},e_{24},e_{34}$ that have $1$ in the $(i,j)$ entry, $-1$ in the $(j,i)$ entry, and $0$ elsewhere ($1\leq i < j \leq 4$). A computation of Lie brackets gives
\[
	[e_{12},e_{13}] = e_{23},\, [e_{12},e_{14}] = e_{24},\, [e_{12},e_{23}] = -e_{13},\, [e_{12},e_{24}] = -e_{14},\, [e_{12},e_{34}] = 0, 
\]
\[
	[e_{13},e_{14}] = e_{34},\,[e_{13},e_{23}] = e_{12},\, [e_{13},e_{24}] = 0,\, [e_{13},e_{34}] = -e_{14},\, [e_{14},e_{23}] = 0,
\]
\[
	[e_{14},e_{24}] = e_{12},\, [e_{14},e_{34}] = e_{13},\, [e_{23},e_{24}] = e_{34}\, [e_{23},e_{34}] = - e_{24},\, [e_{24},e_{34}] = e_{23}.
\]
Let $\XC$ be a linear vector field on $SO(4)$. Let us denote by $\DC = \ad(X)$ the associated derivation to $\XC$, where $X$ is an right invariant vector field on $SO(4)$. Our next step is to describe the derivation $\DC$. To do this, write
\[
	X = a e_{12} +b e_{13} + c e_{14} + d e_{23} + e e_{24} + f e_{34}, \ \ a,b,c,d,e,f \in \R.
\]
By Lie brackets above, we compute  
\[
	\DC = \ad(X) =
	\left(
	\begin{array}{cccccc}
		0 & -d & -e & b & c & 0 \\
		d & 0 & -f & -a & 0 & c \\
		e & f & 0 & 0 & -a & -b \\
		-b & a & 0 & 0 & -f & e \\
		-c & 0 & a & f & 0 & -d \\
		0 & -c & b & -e & d & 0
	\end{array}
	\right).
\]
Some calculus show that eigenvalues of $\DC = \ad(X)$ are 
\[
	 \{ 0 ,\pm \sqrt{-(a+f)^2-(b-e)^2-(c+d)^2}, \pm \sqrt{-(a+f)^2-(b+e)^2-(c-d)^2}\}.
\]
We observe that eigenvalues are according to Theorem \ref{derivation}. We now are in position to give a condition that characterize periodic orbits of an invariant or linear flow. 

\begin{theorem}\label{periodicSO(4)}
	Let $\XC$ be a linear vector field on $SO(4)$. Consider the derivation $\DC = \ad(X)$ of $\XC$, where $X$ is a right invariant vector field such that  
	\[
		X = a e_{12} +b e_{13} + c e_{14} + d e_{23} + e e_{24} + f e_{34}, \ \ a,b,c,d,e,f \in \R.
	\]
	A necessary and sufficient condition to every orbit that is not a fix point of linear flow $\varphi_t$ or invariant flow $\exp(tX)$ be periodic is that $bc=ed$. Furthermore, if there are periodic orbits, then their periods are $T = 2\pi /\sqrt{(a+f)^2+(b+e)^2+(c-d)^2}$. 
\end{theorem}
\begin{proof}. 
	It is a simple application of Theorem \ref{teo1}. In fact, orbits of linear flow $\varphi_t$ or invariant flow $\exp(tX)$ are periodic if, and only if, eigenvalues of $\DC=\ad(X)$ are $0$ or $\pm \alpha i$ for a unique $\alpha \in \R$. Last condition is equivalent to 
	\[
		\sqrt{-(a+f)^2-(b-e)^2-(c+d)^2} = \sqrt{-(a+f)^2-(b+e)^2-(c-d)^2},
	\]
	which is equivalent to $bc = ed$. 
\end{proof}

As direct application of Theorem above, any right invariant vector field $e_{12}$, $e_{13}$, $e_{14}$, $e_{23}$, $e_{24}$ or $e_{34}$ yields periodic orbits for linear or invariant flows.

\end{document}